\documentclass[12pt]{amsart}

\usepackage{amsmath, amssymb, amsthm}
\usepackage{mathtools}
\usepackage{hyperref}
\usepackage{enumitem}
\usepackage{geometry}
\usepackage{xcolor}
\newtheorem{theorem}{Theorem}[section]
\newtheorem{proposition}[theorem]{Proposition}

\newtheorem{corollary}[theorem]{Corollary}
\newtheorem{definition}[theorem]{Definition}
\newtheorem{remark}[theorem]{Remark}

\newtheorem*{ack}{Acknowledgement}

\newcommand{\M}{\mathcal{M}}
\newcommand{\CW}{\mathrm{CW}}

\newcommand{\cW}{\mathcal{W}}
\newcommand{\Wim}{\mathcal{W}^{\mathrm{im}}}
\newcommand{\bR}{\mathbb{R}}

\newcommand{\eps}{\varepsilon}
\newcommand{\CE}{\mathrm{CE}}

\title[Generation of immersed Lagrangians by cocores]{Generation of immersed Lagrangians by cocores}

\author{Wonbo Jeong}
\address[Wonbo Jeong]{Department of Mathematics and Center for Nano Materials\\ Sogang University\\ 35 Baekbeom-ro\\ Mapo-gu\\ Seoul 04107\\ Republic of Korea}
\email{wonbo.jeong@gmail.com}

\author{Dogancan Karabas}
\address[Dogancan Karabas]{Undergraduate Program, Temple University Japan Campus, Tokyo, Japan}
\email{dogancan.karabas@tuj.temple.edu}

\author{Sangjin Lee}
\address[Sangjin Lee]{Korea Institute for Advanced Study, 85 Hoegiro Dongdaemun-gu, Seoul 02455, Republic of Korea}
\email{sangjinlee@kias.re.kr}

\begin{document}
	
\begin{abstract}
	We extend the generation theorem of Chantraine--Dimitroglou
	Rizell--Ghiggini--Golovko to exact Lagrangian immersions in Weinstein
	manifolds. We prove that an exact Lagrangian immersion equipped with an augmentation of the Chekanov--Eliashberg algebra of its Legendrian
	lift, or equivalently, equipped with a corresponding bounding cochain, is generated by the Lagrangian cocores.
\end{abstract}

\maketitle

\tableofcontents

\section{Introduction}
\label{sec:introduction}

Generation by standard Lagrangian objects is a fundamental structural result
for Fukaya categories. The goal of this paper is to prove such a
generation result for exact Lagrangian immersions with bounding cochains
supported on positive-action double points. Equivalently, via the Legendrian
lift, these are exact Lagrangian immersions equipped with augmentations of the
corresponding Chekanov--Eliashberg algebras.

The main result is the following.

\begin{theorem}[Theorem \ref{thm:immersed-generation}, Corollary \ref{cor:embedded-fukaya}]
	\label{thm:intro-generation}
	Let $L$ be an exact Lagrangian immersion in a Weinstein manifold $W$ of
	finite type, equipped with a bounding cochain $b$ supported on the
	positive-action double points of $L$, or equivalently with an
	augmentation of the Chekanov--Eliashberg algebra of its Legendrian lift.
	Then $(L,b)$ is generated by the Lagrangian cocores of $W$ in the
	immersed wrapped Fukaya category $\Wim(W)$ of \cite{gao17}. Consequently,
	$(L,b)$ belongs to the triangulated closure $\mathrm{Tw}\,\cW(W)$ of the
	embedded wrapped Fukaya category $\cW(W)$.
\end{theorem}

Generation results of this kind go back to Abouzaid's theorem \cite{abo11}, stating that a cotangent
fiber generates the wrapped Fukaya category of a cotangent bundle. For
embedded exact Lagrangians in general Weinstein manifolds, the corresponding
generation theorem was proved by
Chantraine--Dimitroglou Rizell--Ghiggini--Golovko \cite{cdrgg17}, and later
by Ganatra--Pardon--Shende \cite{gps2}: they are generated
by the Lagrangian cocores, which can be seen as generalization of cotangent fibers.

The present paper extends this embedded generation result to exact Lagrangian
immersions equipped with positive-action bounding cochains. Immersed Floer
theory has been developed in several related directions: Akaho--Joyce
\cite{Akaho-Joyce10} constructed an immersed Lagrangian Floer theory for
compact symplectic manifolds, Alston--Bao \cite{als-bao18} studied immersed compact Lagrangians in exact symplectic manifolds, and Gao \cite{gao17}
developed an immersed wrapped Floer theory for possibly noncompact exact
Lagrangians in Liouville manifolds using pearly-tree models.

When extending the generation argument of \cite{cdrgg17} to the immersed
setting, one encounters an additional categorical issue. Although
\cite{cdrgg17} defines Floer complexes and Hamiltonian-perturbed operations
for augmented immersed Lagrangians, it does not construct a complete
$A_\infty$-category whose objects are immersed Lagrangians; the authors note
that the required coherent Hamiltonian perturbations have not been verified
\cite[Section~6.2]{cdrgg17}. Thus, their generation argument cannot be applied
directly as a statement inside an immersed wrapped Fukaya category.

We resolve this by working in Gao's immersed wrapped Fukaya category
$\Wim(W)$ \cite{gao17}, which is a well-defined $A_\infty$-category whose
objects are exact Lagrangian immersions equipped with bounding cochains. The
operations in $\Wim(W)$ are defined using stable pearly tree maps, while
the operations in \cite{cdrgg17} are defined by Hamiltonian-perturbed
holomorphic polygons with boundary punctures weighted by augmentation values.
The key point of the paper is that, for bounding cochains supported on
positive-action double points (see Definition \ref{def:bounding-cochain-positive}), these two descriptions agree. More precisely,
the pearly-tree decorations in \cite{gao17} degenerate to the
corresponding Hamiltonian-perturbed polygon counts, and the coefficients of
the bounding cochains agree with the augmentation weights under the natural
correspondence between augmentations and bounding cochains supported on
positive-action double points. Related comparisons between Hamiltonian-perturbed and pearly models of Floer
theory in non-wrapped settings appear in the work of
Alston--Bao \cite{als-bao21} and Zhang \cite{zha25}.

The comparison established in this paper allows us to import the generation argument of
\cite{cdrgg17}, with the necessary modifications, into the well-defined
category $\Wim(W)$. The result is that exact Lagrangian immersions equipped
with positive-action bounding cochains are generated by the Lagrangian
cocores. Consequently, such objects lie in the triangulated closure of the
embedded wrapped Fukaya category inside the immersed wrapped Fukaya category.

The paper is organized as follows. In Section~\ref{subsec:setup}, we recall
basic definitions concerning Weinstein manifolds and exact Lagrangian
immersions with cylindrical end. We also introduce the wrapped Floer
$A_{\infty}$-algebra $\CW^*(L)$ of \cite{gao17}, the wrapped Floer cochain complexes
$\CW^*(L_0,L_1)$, our action conventions for double-point and Reeb chord
generators, and the basic energy inequality used throughout the paper.

In Section~\ref{subsec:rigidity}, we prove the analytic degeneration results
which drive the comparison. First, we show that if all inputs are
positive-action double-point generators, then any rigid stable pearly tree map
contributing to the operations on $\CW^*(L)$ degenerates to a single
holomorphic disk (Theorem~\ref{thm:pearly-single-disk}). We then prove the
corresponding statement for the bimodule operations of \cite{gao17}: any rigid stable broken
Floer trajectory with positive-action double-point inputs degenerates to a
single holomorphic strip with boundary marked points
(Theorem~\ref{thm:bimod-single-strip}).

In Section~\ref{subsec:bc-aug}, we relate the two algebraic inputs. We show
that bounding cochains supported on the positive-action double-point
generators of $\CW^*(L)$ are in natural bijection with augmentations of the
Chekanov--Eliashberg algebra $\CE(L^+)$ of the Legendrian lift
$L^+\subset W\times \bR$ (Proposition~\ref{prop:bc-aug}).

In Section~\ref{subsec:comparison}, we compare the two Floer-theoretic
frameworks. Using the degeneration results and the bounding
cochain--augmentation correspondence, we prove that the $A_\infty$-operations of \cite{gao17}
agree with the corresponding Hamiltonian-perturbed $A_{\infty}$-operations of
\cite{cdrgg17} for the objects considered in this paper
(Theorem~\ref{thm:floer-equivalence}). This gives a well-defined
Hamiltonian-perturbed interpretation of the relevant full subcategory of
$\Wim(W)$ (Corollary~\ref{cor:immersed-fukaya}).

Finally, in Section~\ref{subsec:generation}, we adapt the generation argument
of \cite{cdrgg17} to $\Wim(W)$ and prove Theorem~\ref{thm:intro-generation}
as Theorem~\ref{thm:immersed-generation}. As a consequence, an exact
Lagrangian immersion equipped with a positive-action bounding cochain belongs
to the triangulated closure $\mathrm{Tw}\,\cW(W)$ of the embedded wrapped
Fukaya category (Corollary~\ref{cor:embedded-fukaya}).

\begin{ack}
	During this work, Wonbo Jeong was supported by the National Research Foundation of Korea (NRF) grants funded by the Korean government (MIST, No. RS-2025-24803252) and through the G-LAMP program (MOE, RS-2024-00441954).
	Part of this work was carried out while Dogancan Karabas was at the Kavli Institute for the Physics and Mathematics of the Universe (Kavli IPMU), the University of Tokyo, where the author was supported by the World Premier International Research Center Initiative (WPI), MEXT, Japan. Further parts of this work were carried out at the Kyoto University Institute for Advanced Study (KUIAS), in the Hiraoka Laboratory, and at Temple University, Japan Campus. 
	Sangjin Lee is supported by a KIAS Individual Grant (MG094401) at Korea Institute for Advanced Study.
\end{ack}

\section{Setup}
\label{subsec:setup}

We first recall the definitions for Liouville and Weinstein manifolds. A
\emph{Liouville manifold} is a pair $(W,\theta)$ such that $W$ is a $2n$-dimensional smooth manifold, $\omega=d\theta$ is a symplectic form, and the Liouville vector field $Z$, defined by
\[
\iota_Z\omega=\theta,
\]
generates a complete expanding flow. A \emph{Weinstein manifold} is a
Liouville manifold $(W,\theta)$ equipped with an exhausting Morse function
$\phi\colon W\to \mathbb R$ such that the Liouville vector field $Z$ is
gradient-like for $\phi$. We say that $W$ is of \emph{finite type} if $\phi$
has finitely many critical points. Throughout the paper, all Weinstein
manifolds are assumed to be of finite type.

We now fix a compact Weinstein domain $W_0$ whose completion is $W$. In other
words,
$W_0$ is a compact exact symplectic manifold with boundary, the Liouville
vector field points outward along $\partial W_0$, and the Liouville flow
identifies the complement of the interior of $W_0$ with the positive
symplectization
\[
\partial W_0 \times [1,\infty).
\]
On this end, we write $r$ for the radial coordinate and have
\[
\theta = r\alpha,
\qquad
\alpha := \theta|_{\partial W_0}.
\]
Hence
\[
W = W_0 \cup_{\partial W_0}
\bigl(\partial W_0\times [1,\infty)\bigr).
\]
We call $W_0$ the \emph{interior part} of $W$ and
$\partial W_0\times [1,\infty)$ the \emph{cylindrical end}. The contact
manifold $\partial W_0$, equipped with the contact form $\alpha$, is denoted
by $\partial_\infty W$ and called the \emph{ideal contact boundary} of $W$.

The \emph{skeleton}, or \emph{core}, of $W$ is defined by
\[
\operatorname{Skel}(W,\theta)
:=
\bigcap_{t>0} Z^{-t}(W_0),
\]
where $Z^t$ denotes the time-$t$ flow of the Liouville vector field. Thus the
skeleton consists of the points which do not escape to the cylindrical end
under the positive Liouville flow. The \emph{(Lagrangian) cocores} of $W$ are the properly embedded exact Lagrangian
disks dual to the $n$-dimensional strata of the skeleton. Locally, near
a smooth point of such a stratum, a cocore is modeled on a cotangent fiber
dual to the stratum.

An \emph{exact Lagrangian immersion (with cylindrical end)} is a proper
immersion $i \colon L \rightarrow W$ of an $n$-dimensional manifold
such that $i^*\theta = df$ for a smooth function $f \colon L \to \bR$
called the \emph{potential} (or the \emph{primitive function}). We assume that the immersion is an embedding outside
finitely many points, all self-intersection points are transverse double
points, and $i(L)$ is tangent to the
Liouville vector field $Z$ on the cylindrical end of $W$. We sometimes refer to $i(L)$
by $L$ by abuse of notation and call it an immersed Lagrangian. We assume throughout that each double point $x$ of $L$ with preimage $i^{-1}(x) = \{x^+, x^-\}$ satisfies
\[
  f(x^+) > f(x^-).
\]

\begin{remark}\label{rmk:nice-immersion}
	Following \cite{cdrgg17}, we call an exact Lagrangian immersion 
	\emph{nice} if $f(x^+) \neq f(x^-)$ at every double point $x$. 
	Every exact Lagrangian immersion can be perturbed by a Lagrangian 
	isotopy to a nice one, and the convention 
	$f(x^+) > f(x^-)$ is then achieved by labeling the two preimages 
	accordingly.
\end{remark}


Following \cite{gao17}, we denote by $\CW^*(L)$ the wrapped Floer
$A_{\infty}$-algebra associated to the immersed Lagrangian $i\colon L\to W$ and a
fixed Hamiltonian $H$ quadratic at infinity, meaning that
\[
H=r^2
\]
on the cylindrical end of $W$. It is the free module
generated by the following two kinds of generators:
The first kind consists of critical points of an auxiliary Morse function on
the fiber product
\[
L \times_i L
=
\{(p,q)\in L\times L \mid i(p)=i(q)\}
=
\Delta_L
\sqcup
S(L,i),
\]
where
\[
S(L,i):=\{(p,q)\in L\times L \mid p\neq q,\ i(p)=i(q)\}.
\]
In \cite{gao17}, these generators are equipped with capping half-disks.
We suppress this capping data from the notation, since the arguments in this
paper do not depend on the choice of capping data. The second kind consists of
non-constant time-one Hamiltonian chords of $H$.

Concretely, the generators of $\CW^*(L)$ fall into three types:
\begin{itemize}
	\item \emph{Morse generators}: critical points on the diagonal
	$\Delta_L \cong L$, which correspond to constant chords at Morse
	critical points of $L$,
	\item \emph{double-point generators}: since $S(L,i)$ is discrete, every point is a critical point,
	corresponding to a double point of $L$; each double point $x$
	contributes two generators
	\[
	x_+ := (x^+, x^-), \qquad x_- := (x^-, x^+).
	\]
	We define their \emph{actions} by
	\[
	a(x_+) := f(x^+) - f(x^-), \qquad
	a(x_-) := f(x^-) - f(x^+).
	\]
	By our convention $f(x^+)>f(x^-)$, the generator $x_+$ has positive
	action and $x_-$ has negative action.
	
	\item \emph{Reeb chord generators}: non-constant chords, which
	correspond to Reeb chords of the Legendrian $\partial_\infty L$
	on the ideal contact boundary $\partial_\infty W$. If the corresponding time-one Hamiltonian chord $x$
	lies in $\partial W_0\times\{r\}$ on the cylindrical end, we define its \emph{action} by
	\[
	a(x):=-r^2.
	\]
	Thus Reeb chord generators have negative action; see also the passage following
	\cite[Lemma~3.3]{gao17} and the proof of \cite[Lemma~7.23]{gao17}.
\end{itemize}

	Throughout, we assume that each immersed Lagrangian $i \colon L \to W$ 
	is graded. The chosen 
	grading determines a grading for each generator of $\CW^*(L)$; 
	in particular, the degree of each double-point generator $x_\pm$ and 
	each Reeb chord generator is their Maslov index, and the degree of a 
	Morse generator is its Morse index. We refer to \cite[Section~4.2]{gao17} 
	for the precise definitions.

Consider a pseudoholomorphic disk $u$ with boundary on $i(L)$ such that
the boundary of $u$ admits a continuous lift to $L$ away from its boundary
punctures. Assume that at each boundary puncture, the map is asymptotic either
to a point of $i(L)$, which may be a double point of the immersion, or to a
Reeb chord.

We distinguish one boundary puncture as the output and regard the remaining
boundary punctures as inputs. If a puncture is asymptotic to a double point of
the immersion, then the corresponding ordered branch jump is determined by
this input-output convention: at the output puncture, it is read in the
direction induced by the boundary orientation, while at an input puncture, it
is read in the opposite direction. Similarly, if a puncture is asymptotic to a Reeb chord, the chord runs from
the incoming one-sided lift to the outgoing one-sided lift at the output
puncture, and from the outgoing one-sided lift to the incoming one-sided lift
at an input puncture, where incoming and outgoing are taken with respect to
the boundary orientation.

With these conventions, and setting $a(z)=0$ when the asymptotic object $z$ is
neither a double point nor a Reeb chord, Stokes' theorem gives the energy identity
\[
E(u):=\int u^*\omega
= a(y)-\sum_{j=1}^k a(x_j),
\]
where $y$ denotes the asymptotic object at the output puncture and
$x_1,\ldots,x_k$ denote those at the input punctures. Since $E(u)\geq 0$, with
equality only for constant maps, any rigid non-constant pseudoholomorphic disk
contributing to the $A_\infty$-operations must satisfy
\begin{equation}
	\label{eq:energy}
	a(y)>\sum_{j=1}^k a(x_j).
\end{equation}


For two exact Lagrangian immersions $i_0 \colon L_0 \to W$ and
$i_1 \colon L_1 \to W$ intersecting transversely away from their
double points, both the framework of \cite{gao17} and that of
\cite{cdrgg17} define the generators of the wrapped Floer cochain complex $\CW^*(L_0, L_1)$
as time-one Hamiltonian chords from $i_0(L_0)$ to
$i_1(L_1)$, where intersection points in the interior of $W$ are
viewed as constant chords. We therefore use the same notation $\CW^*(L_0,L_1)$ for
the underlying graded module in both frameworks.

\section{Degeneration of pearly trees and Floer trajectories}
\label{subsec:rigidity}

The key analytic input for the equivalence between the frameworks of \cite{gao17} and \cite{cdrgg17} is that
the pearly-tree operations of \cite{gao17} simplify drastically when all input
generators carry positive action. We make this precise in the
following two theorems, which show that every rigid pearly tree map
degenerates to a single holomorphic disk, and every rigid Floer
trajectory degenerates to a single holomorphic strip. Together with
the bijection of Section~\ref{subsec:bc-aug}, these results imply
that the pearly-tree decorations in the framework of \cite{gao17} reduce to simple
boundary punctures weighted by augmentation values, making the
$A_\infty$-operations of \cite{gao17} identical to augmentation-weighted
operations of \cite{cdrgg17} on the objects of interest.

Recall from \cite[Definition~4.18]{gao17} that a \emph{stable pearly
	tree map} $\mathcal T$ associated to an underlying colored rooted tree $T$ in the
sense of \cite[Definition~4.7]{gao17} consists of:
\begin{itemize}
	\item for each vertex of $T$, a pseudoholomorphic disk with boundary on $i(L)$, such that the boundary admits a continuous lift to $L$ away from the boundary punctures, each of which corresponds to an adjacent edge,
	\item for each interior edge of $T$ connecting two vertices, a gradient flow trajectory on
	$L \times_i L$ of finite length,
	\item for each leaf edge (a ray from a vertex toward an input),
	a gradient flow ray on $L \times_i L$ asymptoting to a
	Morse-type or double-point input generator, or a
	pseudoholomorphic strip asymptoting to a Hamiltonian chord
	input generator, according to the color of the edge,
	\item for the root edge (a ray from a vertex toward the output),
	a gradient flow ray on $L \times_i L$ asymptoting to a
	Morse-type or double-point output generator, or a
	pseudoholomorphic strip asymptoting to a Hamiltonian chord
	output generator, according to the color of the edge,
\end{itemize}
together with matching conditions at each vertex relating the endpoints
of the adjacent flow trajectories or strip-like ends to the boundary
punctures of the adjacent disk components. In the degenerate case where $T$ has no vertices and consists of a
single edge, $\mathcal{T}$ consists of a single gradient flow
trajectory from the input to the output with no disk component.

The $A_\infty$-structure maps $m_k$ on $\CW^*(L)$ are defined by
counting rigid stable pearly tree maps, subject to the degree constraint:
\begin{equation}
	\label{eq:pearly-degree}
	|c_0| - \sum_{j=1}^k |c_j| = 2 - k,
\end{equation}
where $c_0$ is the root (output) generator and $c_1, \ldots, c_k$ are
the leaf (input) generators.

\begin{theorem}[Degeneration of pearly trees]
	\label{thm:pearly-single-disk}
	Let $\mathcal{T}$ be a rigid stable pearly tree map contributing to the $A_\infty$-operation $m_k$ on
	$\CW^*(L)$. If all $k$ input generators $c_1, \ldots, c_k$ are
	positive-action double-point generators, i.e.\ $a(c_j) > 0$ for
	all $j = 1, \ldots, k$, then:
	\begin{enumerate}[label=(\roman*)]
		\item The output generator $c_0$ is a positive-action double-point
		generator, $a(c_0) > 0$.
		\item The underlying tree $T$ has a single internal vertex, i.e.\
		$\mathcal{T}$ consists of a single holomorphic disk component
		with all edges of $T$ given by constant Morse trajectories.
	\end{enumerate}
\end{theorem}

\begin{proof}
	
	We proceed in two steps.
	
	\medskip
	\noindent\textbf{Step 1: All gradient trajectories are constant.}
	
	First observe that the only generators of positive action are the
	positive-action double-point generators. Also, the set of double-point generators $S(L,i)$ is discrete, hence any gradient trajectory converging to a point
	in $S(L,i)$ is constant.
	
	Orient $T$ as a rooted tree, with all edges directed toward the root edge, so
	that each internal vertex has a single outgoing edge and zero or more
	incoming edges, which we regard as the output and inputs of the
	holomorphic disk at that vertex respectively. Order all edges of $T$
	by their distance to the root edge, and proceed by induction on this
	distance starting from the largest.
	
	Edges at maximal distance are either leaf edges, asymptoting to
	positive-action double-point generators by assumption, or interior
	edges adjacent to valency-one vertices, whose disk components have
	no inputs and hence have positive-action output by the energy
	inequality~\eqref{eq:energy}. In either case the associated gradient
	trajectory converges to a point in $S(L,i)$, hence is constant.
	
	Suppose by induction that all edges at distance greater than $d$ from
	the root carry positive-action double-point generators and have
	constant associated gradient trajectories. Consider an edge $e$ at
	distance $d$. If $e$ is a leaf edge or adjacent to a valency-one
	vertex, the same argument as the base case applies. Otherwise, the
	source vertex of $e$ has all incoming edges at distance greater than
	$d$, hence carrying positive-action double-point generators by the
	induction hypothesis. By the energy inequality~\eqref{eq:energy},
	the output of the holomorphic disk at the source vertex also has
	positive action, hence lies in $S(L,i)$, and the gradient trajectory
	associated to $e$ converges to a point in $S(L,i)$, hence is constant.
	
	By induction, all gradient trajectories are constant and all
	generators in the tree lie in $S(L,i)$ and have positive action. In
	particular the root edge carries a positive-action double-point
	generator, proving (i). The pearly tree $\mathcal{T}$ thus degenerates
	to a nodal configuration of $m \geq 0$ disk components with adjacent
	components sharing a boundary positive-action double-point puncture.

	\medskip
	\noindent\textbf{Step 2: $m = 1$.}
	
	The moduli space of such a configuration decomposes as a disjoint
	union over all assignments of internal positive-action double-point
	generators, and for each such assignment the moduli space is a product
	\[
	\M(D_1) \times \cdots \times \M(D_m),
	\]
	where $D_i$ denotes the $i$-th disk component with its fixed input
	and output generators. Since the internal matching generators lie in the discrete set
	$S(L,i)$, the moduli space of such configurations is a finite union
	of products of the component moduli spaces. Hence its virtual
	dimension is the sum of the virtual dimensions of the disk
	components. By standard transversality considerations, rigidity of
	the total configuration implies that each nonempty component factor
	is rigid, i.e.\ the virtual dimension of $\M(D_i)$ is zero for all $i$. Applying the degree constraint~\eqref{eq:pearly-degree} to
	each disk $D_i$ with input generators $c_{i,1}, \ldots, c_{i,n_i}$
	and output generator $c_{i,0}$ gives
	\[
	|c_{i,0}| - \sum_{j=1}^{n_i} |c_{i,j}| = 2 - n_i.
	\]
	Summing over $i = 1, \ldots, m$:
	\[
	\sum_{i=1}^m |c_{i,0}| - \sum_{i=1}^m \sum_{j=1}^{n_i}
	|c_{i,j}| = 2m - \sum_{i=1}^m n_i.
	\]
	The $m-1$ internal generators each appear once as an output $c_{i,0}$
	and once as an input $c_{i,j}$, and hence cancel on the left-hand side.
	The remaining terms are the external output $c_0$ and the $k$ external 
	inputs $c_1, \ldots, c_k$. The total number of inputs across all
	disks is $\sum_{i=1}^m n_i = k + (m-1)$, so
	\begin{equation}\label{eq:nodal-configuration}
	|c_0| - \sum_{j=1}^k |c_j| = 2m - k - (m-1) = m + 1 - k.
	\end{equation}
	Comparing with the global constraint~\eqref{eq:pearly-degree} gives
	$m + 1 - k = 2 - k$, hence $m = 1$, proving (ii).
\end{proof}

We now extend Theorem~\ref{thm:pearly-single-disk} to the bimodule
structure maps
\[
n_{k,l} : \CW^*(L_1)^{\otimes l} \otimes
\CW^*(L_0, L_1) \otimes
\CW^*(L_0)^{\otimes k}
\longrightarrow \CW^*(L_0, L_1)
\]
defined in \cite[Definitions~4.29-4.33]{gao17} by counting rigid
stable broken Floer trajectories. A \emph{stable unbroken Floer
	trajectory} consists of:
\begin{itemize}
	\item a holomorphic strip $u : \mathbb{R} \times [0,1] \to W$
	with boundary on $i_0(L_0) \cup i_1(L_1)$, asymptotic to chords
	$c^{\rightarrow}$ (input) and $c^{\leftarrow}$ (output) at $s \to +\infty$ and $s \to -\infty$
	respectively,
	\item a finite collection of boundary marked points on
	$\mathbb{R} \times \{0\}$ and $\mathbb{R} \times \{1\}$,
	\item for each boundary marked point on $\mathbb{R} \times \{0\}$
	(resp.\ $\mathbb{R} \times \{1\}$), a stable pearly tree map associated to $L_0$ (resp.\ $L_1$) attached at that
	point via its root (output) generator, with leaf (input) generators
	contributing to the external inputs $c^0_1, \ldots, c^0_k$
	(resp.\ $c^1_1, \ldots, c^1_l$) of $n_{k,l}$,
\end{itemize}
together with matching conditions at each boundary marked point
relating the output generator of the attached pearly tree to the
boundary value of the strip. A \emph{stable broken Floer trajectory}
is a sequence $(\widetilde{\Sigma}^{(1)}, \ldots,
\widetilde{\Sigma}^{(K)})$ of stable unbroken Floer trajectories with
matching asymptotic conditions $c^{\leftarrow}(\nu) = c^{\rightarrow}(\nu+1)$ for $\nu =
1, \ldots, K-1$.

The degree constraint for a rigid broken trajectory is:
\begin{equation}
	\label{eq:bimod-degree}
	|c^{\leftarrow}| - |c^{\rightarrow}| - \sum_{a=1}^k |c^0_a|
	- \sum_{b=1}^l |c^1_b| = 1 - k - l.
\end{equation}

\begin{theorem}[Degeneration of Floer trajectories]
	\label{thm:bimod-single-strip}
	Let $(\widetilde{\Sigma}^{(1)}, \ldots, \widetilde{\Sigma}^{(K)})$
	be a rigid stable broken Floer trajectory contributing to the
	bimodule operation $n_{k,l}$, with inputs $c^0_k, \ldots, c^0_1,
	c^{\rightarrow}, c^1_1, \ldots, c^1_l$ and output $c^{\leftarrow}$. If each external
	leaf generator $c^i_j$ of every attached pearly tree is a
	positive-action double-point generator, then:
	\begin{enumerate}[label=(\roman*)]
		\item Each attached stable pearly tree map degenerates to one whose underlying tree has no internal vertices and
		consists of a single edge with a constant gradient trajectory.
		\item The broken trajectory is unbroken, i.e.\ $K = 1$.
	\end{enumerate}
	Consequently, the entire configuration is a single holomorphic
	strip with boundary marked points corresponding to positive-action
	double-point generators.
\end{theorem}

\begin{proof}
	We proceed in two steps.
	
	\medskip
	\noindent\textbf{Step 1: All gradient trajectories are constant.}
	
	By the same argument as in Step 1 of the proof of
	Theorem~\ref{thm:pearly-single-disk}, applied to each attached
	stable pearly tree map, all gradient trajectories in every attached
	pearly tree are constant, since all external leaf generators are
	positive-action double-point generators by assumption. The entire
	configuration thus degenerates to a nodal configuration consisting
	of $K$ holomorphic strips decorated by $m^0$ and $m^1$ holomorphic
	disk components attached at boundary marked points on
	$\mathbb{R} \times \{0\}$ and $\mathbb{R} \times \{1\}$
	respectively.
	
	\medskip
	\noindent\textbf{Step 2: $K = 1$ and $m^0 = m^1 = 0$.}
	
	Let $M = K + m^0 + m^1$ denote the total number of holomorphic
	strip and disk components in the nodal configuration. The moduli
	space decomposes as a product, and for the configuration to be rigid
	each factor must be rigid. Applying the degree constraint to each
	component and summing, the internal generators cancel exactly as in
	Step 2 of Theorem~\ref{thm:pearly-single-disk}, giving
	\[
	|c^{\leftarrow}| - |c^{\rightarrow}| - \sum_{a=1}^k |c^0_a|
	- \sum_{b=1}^l |c^1_b| = M + 1 - (1 + k + l) = M - k - l,
	\]
	analogous to \eqref{eq:nodal-configuration}.
	Comparing with the global constraint~\eqref{eq:bimod-degree} gives
	$1 - k - l = M - k - l$, hence $M = 1$. Since $M = K + m^0 + m^1$, we conclude $K = 1$ and $m^0 = m^1 =
	0$, proving (i) and (ii). Consequently, the configuration is a
	single holomorphic strip with boundary marked points at
	positive-action double-point generators, with no disk components
	attached.
\end{proof}

\section{Bounding cochains and augmentations}
\label{subsec:bc-aug}

Recall that the \emph{Legendrian lift} of $L$ is
\[
L^+ := \{(i(p), -f(p)) \mid p \in L\} \subset W \times \bR_z,
\]
equipped with the contact form $dz + \theta$. The Reeb chords of
$L^+$ are in bijection with the positive-action double-point
generators of $\CW^*(L)$: the positive-action double-point
generator $x_+ = (x^+, x^-)$ with $f(x^+) > f(x^-)$ corresponds to
a Reeb chord $\bar{x}$ of length $a(x_+) = f(x^+) - f(x^-)$, with
degree $|\bar{x}| = 1-|x_+|$. The \emph{Chekanov--Eliashberg
	algebra} $\CE(L^+)$ is the free unital noncommutative DGA generated
by these Reeb chords.

\begin{definition}\label{def:bounding-cochain-positive}
	A bounding cochain $b\in \CW^1(L)$ for $L$ is \emph{supported on the
		positive-action double points} if $b$ is a linear combination of
	positive-action double-point generators $x_+$ in $\CW^1(L)$.
\end{definition}

\begin{proposition}
	\label{prop:bc-aug}
	There is a natural bijection between bounding cochains supported on
	the positive-action double points of $L$ and augmentations of
	$\CE(L^+)$.
\end{proposition}

\begin{proof}
	Let $b = \sum_x \lambda_x \cdot x_+ \in \CW^1(L)$ be a bounding
	cochain supported on positive-action double points, where the sum
	runs over all degree-$1$ positive-action double-point generators
	$x_+$ and $\lambda_x \in k$. Since $b$ is a bounding cochain, it has to satisfy the Maurer--Cartan equation, which reads
	\[
	\sum_{k \geq 0} m_k(b, \ldots, b) = 0.
	\]
	By Theorem~\ref{thm:pearly-single-disk}, every rigid pearly tree map
	with positive-action double-point inputs degenerates to a single
	holomorphic disk with positive-action double-point output. At any such output $y_+$, the
	coefficient in $\sum_k m_k(b, \ldots, b)$ is
	\begin{equation}
		\label{eq:mc}
		\#\M(y_+) + \sum_{x_1} \lambda_{x_1} \#\M(y_+; (x_1)_+)
		+ \sum_{x_1, x_2} \lambda_{x_1} \lambda_{x_2}
		\#\M(y_+; (x_1)_+, (x_2)_+) + \cdots = 0,
	\end{equation}
	where $\#\M(y_+; (x_1)_+, \ldots, (x_d)_+)$ denotes the signed count
	of rigid pseudoholomorphic disks with boundary on $i(L)$, output
	corner at $y_+$ and input corners at $(x_1)_+, \ldots, (x_d)_+$,
	equivalently by Theorem~\ref{thm:pearly-single-disk} the count of
	rigid stable pearly tree maps with those generators.
	
	We define $\eps_b \colon \CE(L^+) \to k$ by $\eps_b(\bar{x}) =
	\lambda_x$ on degree-$0$ generators $\bar{x}$ and $\eps_b = 0$ on
	generators of nonzero degree. By definition \cite{cdrgg17}, the
	differential of $\CE(L^+)$ on the degree-$(-1)$ generator $\bar{y}$
	is
	\[
	d\bar{y} = \#\M(y_+) + \sum_{x_1} \#\M(y_+; (x_1)_+)\, \bar{x}_1
	+ \sum_{x_1, x_2} \#\M(y_+; (x_1)_+, (x_2)_+)\,
	\bar{x}_1 \bar{x}_2 + \cdots.
	\]
	Applying $\eps_b$ and using \eqref{eq:mc} gives $\eps_b(d\bar{y}) =
	0$, so $\eps_b$ is an augmentation.
	
	Conversely, given an augmentation $\eps \colon \CE(L^+) \to k$,
	define
	\[
	b_{\eps} := \sum_x \eps(\bar{x})\, x_+ \;\in\; \CW^1(L),
	\]
	where the sum runs over all degree-$1$ positive-action double-point
	generators $x_+$. Since $b_{\eps}$ is supported on positive-action
	double-point generators, by Theorem~\ref{thm:pearly-single-disk}
	every rigid pearly tree map with inputs from $b_{\eps}$ degenerates to a
	single holomorphic disk with positive-action double-point output. So the
	Maurer--Cartan equation has contributions only at positive-action
	double-point generators $y_+$, and it
	reduces to
	\[
	\#\M(y_+) + \sum_{x_1} \eps(\bar{x}_1)\, \#\M(y_+; (x_1)_+)
	+ \sum_{x_1, x_2} \eps(\bar{x}_1)\eps(\bar{x}_2)\,
	\#\M(y_+; (x_1)_+, (x_2)_+) + \cdots
	\;=\; \eps(d\bar{y}),
	\]
	which vanishes since $\eps$ is an augmentation. The maps $b \mapsto
	\eps_b$ and $\eps \mapsto b_{\eps}$ are inverse by construction.
\end{proof}

\begin{remark}
	The notion of a bounding cochain supported on positive-action double 
	points depends on the choice of perturbation of $L$ to a nice immersion 
	in the sense of Remark~\ref{rmk:nice-immersion}. For a fixed nice 
	immersion, the set of such bounding cochains is well-defined. Moreover, 
	by \cite[Proposition~4.7]{cdrgg17}, if two nice immersions are connected by a safe isotopy, i.e.\ an isotopy through nice 
	immersions with only transverse double points, then there is a natural 
	bijection between their augmentations of the Chekanov--Eliashberg algebra. 
	Via Proposition~\ref{prop:bc-aug}, this gives a natural bijection between 
	their positive-action bounding cochains. Thus the set of positive-action bounding cochains is independent of the choice of nice immersion up to safe isotopy.
\end{remark}

\section{Equivalence of the two Floer-theoretic frameworks}
\label{subsec:comparison}

We now make precise the claim that the frameworks of \cite{gao17} and \cite{cdrgg17} produce identical Floer theory for the objects of interest.
By Proposition~\ref{prop:bc-aug}, a bounding cochain $b$ supported
on positive-action double points corresponds canonically to an
augmentation $\eps_b$ of $\CE(L^+)$.

The differential on $\CW^*\!\left((L_0,b_0), (L_1,b_1)\right)$ in
\cite{gao17} is defined by \cite[equation~(4.59)]{gao17} as
\[
m_1^{\mathrm{Gao}}(c^{\rightarrow})
\;:=\; \sum_{k, l \geq 0} n_{k,l}\bigl(
\underbrace{b_1, \ldots, b_1}_{l},\,
c^{\rightarrow},\,
\underbrace{b_0, \ldots, b_0}_{k}
\bigr),
\]
where $b_j$ are bounding cochains for $L_j$ satisfying the
Maurer--Cartan equation.

More generally, as explained in \cite[Section~5.7]{gao17}, the higher
$A_{\infty}$-operations $m_d^{\mathrm{Gao}}$ are defined by counting pseudoholomorphic polygons
with boundary on the relevant immersed Lagrangians, where the boundary segment
labeled by each Lagrangian may carry attached stable pearly trees whose inputs
come from the corresponding bounding cochain.

The differential on
$\CW^*\!\left((L_0, \eps_0), (L_1, \eps_1)\right)$ in the
framework of \cite{cdrgg17}, where $\eps_j \colon \CE(L_j^+) \to k$ are augmentations
of the Chekanov--Eliashberg algebra of $L_j^+$, is defined by \cite[equation~(10)]{cdrgg17} as
\begin{align*}
&m_1^{\mathrm{CDRGG}}(c^{\rightarrow})
\;:=\; \\
&\sum_{c^{\leftarrow}}
\sum_{\substack{x^0_1, \ldots, x^0_k \\ x^1_1, \ldots, x^1_l}}
\#\M\bigl(c^{\leftarrow};\, (x^0_k)_+, \ldots, (x^0_1)_+,\,
c^{\rightarrow},\, (x^1_1)_+, \ldots, (x^1_l)_+\bigr)
\cdot \prod_a \eps_0(\overline{x^0_a})
\cdot \prod_b \eps_1(\overline{x^1_b}) \cdot c^{\leftarrow},
\end{align*}
where the sum runs over all output chords $c^{\leftarrow}$ of degree
$|c^{\leftarrow}| = |c^{\rightarrow}| + 1$ and all collections of double points $x^0_a$ of $L_0$ and
$x^1_b$ of $L_1$, and $\#\M(c^{\leftarrow}; \ldots)$ denotes the signed count of
rigid holomorphic strips with output $c^{\leftarrow}$, input
$c^{\rightarrow}$, and boundary punctures at the indicated
double-point generators.

More generally, as explained in \cite[Section~6.2]{cdrgg17}, the higher
$A_{\infty}$-operations $m_d^{\mathrm{CDRGG}}$ are defined by counting Hamiltonian-perturbed
pseudoholomorphic polygons with boundary on the relevant immersed Lagrangians. In the immersed case, these polygons are allowed
to have boundary punctures at double points, and these punctures are weighted
by the corresponding augmentation values.

\begin{remark}\label{rmk:immersed-wrapped-fukaya-category}
	We emphasize that \cite{cdrgg17} does not verify that their construction
	gives a complete $A_\infty$-category in the immersed wrapped setting. 
	As noted in \cite[Section~6.2]{cdrgg17}, the coherent Hamiltonian 
	perturbations required to establish the $A_\infty$-relations are not 
	verified there.
	Instead, the authors only use the fact that the component of the wrapped
	$A_\infty$-operation whose inputs come from the ordinary Floer complexes,
	rather than the wrapped Floer complexes, agrees with the corresponding
	operation defined in the unwrapped setting.
	In contrast, the immersed wrapped Fukaya category $\Wim(W)$ defined by
	\cite{gao17} is a well-defined $A_\infty$-category whose objects are exact
	Lagrangian immersions equipped with bounding cochains, and whose $A_{\infty}$-operations
	are $m_d^{\mathrm{Gao}}$.
	We verify the $A_\infty$-relations for the 
		construction of \cite{cdrgg17} in this section by comparing it with 
		the construction of \cite{gao17}.
\end{remark}

\begin{theorem}[Equivalence of $A_\infty$-operations]
	\label{thm:floer-equivalence}
	Let $d\geq 1$, and let $i_j \colon L_j \to W$, $j=0,\ldots,d$, be exact Lagrangian
	immersions which intersect transversely away from their double points.
	Suppose that each $L_j$ is equipped with a bounding cochain $b_j$
	supported on positive-action double points, and let
	$\eps_j=\eps_{b_j}$ be the corresponding augmentation via
	Proposition~\ref{prop:bc-aug}. Then, the $A_{\infty}$-operation
	\[
	m_d^{\mathrm{Gao}}
	\colon
	\CW^*((L_{d-1},b_{d-1}),(L_d,b_d))\otimes\cdots\otimes
	\CW^*((L_0,b_0),(L_1,b_1))
	\to
	\CW^*((L_0,b_0),(L_d,b_d))
	\]
	coincides with the corresponding operation
	$m_d^{\mathrm{CDRGG}}$ after identifying $b_j$ with $\eps_j$.
\end{theorem}


\begin{proof}
	As noted in Section~\ref{subsec:setup}, the underlying graded modules
	coincide: the graded module $\CW^*((L_0,b_0),(L_1,b_1))$ is generated by time-one
	Hamiltonian chords from $i_0(L_0)$ to $i_1(L_1)$, which are exactly the
	generators of
	$\CW^*((L_0,\eps_0),(L_1,\eps_1))$.
	It remains to compare the $A_{\infty}$-operations.
	
	We first compare the differentials and show $m_1^{\mathrm{Gao}} = m_1^{\mathrm{CDRGG}}$. By Theorem~\ref{thm:bimod-single-strip},
	since all inputs of $b_0$ and $b_1$ are positive-action
	double-point generators, each term $n_{k,l}(b_1^{\otimes l},
	c^{\rightarrow}, b_0^{\otimes k})$ in the sum defining
	$m_1^{\mathrm{Gao}}$ counts a single holomorphic strip with $k$
	boundary marked points on $\mathbb{R} \times \{0\}$ at
	positive-action double-point generators $(x^0_1)_+, \ldots,
	(x^0_k)_+$ of $L_0$, and $l$ boundary marked points on
	$\mathbb{R} \times \{1\}$ at positive-action double-point generators
	$(x^1_1)_+, \ldots, (x^1_l)_+$ of $L_1$. The coefficient of each
	such generator $(x^j_a)_+$ in the bounding cochain $b_j$ is
	$\lambda_{x^j_a}$, which by the proof of Proposition~\ref{prop:bc-aug} equals
	the augmentation value $\eps_j(\overline{x^j_a})$. Therefore each
	term contributes
	\[
	\#\M\bigl(c^{\leftarrow};\, (x^0_k)_+, \ldots, (x^0_1)_+,\,
	c^{\rightarrow},\, (x^1_1)_+, \ldots, (x^1_l)_+\bigr)
	\cdot \prod_a \eps_0(\overline{x^0_a})
	\cdot \prod_b \eps_1(\overline{x^1_b})
	\]
	to the count of strips with output $c^{\leftarrow}$. Summing over
	all $k, l \geq 0$ and all collections of double-point generators
	gives
	\begin{align*}
	&m_1^{\mathrm{Gao}}(c^{\rightarrow})
	= \\
	&\sum_{c^{\leftarrow}}
	\sum_{\substack{x^0_1, \ldots, x^0_k \\ x^1_1, \ldots, x^1_l}}
	\#\M\bigl(c^{\leftarrow};\, (x^0_k)_+, \ldots, (x^0_1)_+,\,
	c^{\rightarrow},\, (x^1_1)_+, \ldots, (x^1_l)_+\bigr)
	\cdot \prod_a \eps_0(\overline{x^0_a})
	\cdot \prod_b \eps_1(\overline{x^1_b}) \cdot c^{\leftarrow},
	\end{align*}
	which equals $m_1^{\mathrm{CDRGG}}(c^{\rightarrow})$ by definition,
	as required.
	
	A similar argument applies to the higher operations. Namely,
	Theorem~\ref{thm:pearly-single-disk} and the analogous polygon version
	of Theorem~\ref{thm:bimod-single-strip} imply that, when the inserted
	bounding cochains are supported on positive-action double points, (the polygon
	version of) the Floer trajectories appearing in $m_d^{\mathrm{Gao}}$ reduce to single
	Hamiltonian-perturbed holomorphic polygons with additional boundary
	punctures at positive-action double-point generators. The coefficients
	of these insertions are again identified with the corresponding
	augmentation values by Proposition~\ref{prop:bc-aug}. Hence each
	summand in $m_d^{\mathrm{Gao}}$ agrees with the corresponding summand
	in $m_d^{\mathrm{CDRGG}}$ for every $d\geq 1$.
\end{proof}



Remark~\ref{rmk:immersed-wrapped-fukaya-category} explains why we use the
immersed wrapped Fukaya category $\Wim(W)$ defined by \cite{gao17} as the
ambient $A_\infty$-category for the generation result in
Section~\ref{subsec:generation}. Together with
Theorem~\ref{thm:floer-equivalence}, the existence of this category gives the
following:

\begin{corollary}
	\label{cor:immersed-fukaya}
	The wrapped Fukaya category of exact Lagrangian immersions $(L, b)$
	equipped with bounding cochains supported on positive-action double
	points, whose Floer theory is defined via Hamiltonian perturbed
	Floer theory, is well-defined. Moreover, its $A_\infty$-operations
	coincide with those of \cite{gao17} and \cite{cdrgg17}.
\end{corollary}

\begin{proof}
	The full subcategory of the immersed wrapped Fukaya category
	$\Wim(W)$ spanned by pairs $(L,b)$, where $b$ is a bounding cochain
	supported on positive-action double points, is well-defined by
	\cite{gao17}. Although its $A_\infty$-operations are defined via stable
	pearly tree maps, Theorem~\ref{thm:floer-equivalence} shows that these
	operations agree with the corresponding Hamiltonian-perturbed polygon
	operations of \cite{cdrgg17}.
\end{proof}

\section{The generation theorem}
\label{subsec:generation}

The purpose of this section is to adapt the generation argument of
\cite{cdrgg17} to the immersed wrapped Fukaya category $\Wim(W)$. Using the
identification established in Corollary~\ref{cor:immersed-fukaya}
between the $A_\infty$-operations of \cite{gao17} and those of \cite{cdrgg17}, we show
that the generation result extends to exact Lagrangian immersions
equipped with bounding cochains supported on positive-action double
points. While the overall strategy follows that of \cite{cdrgg17}, several
modifications are required to account for the immersed setting.

\begin{theorem}
	\label{thm:immersed-generation}
	Let $L$ be an exact Lagrangian immersion in a Weinstein manifold
	$W$ of finite type, equipped with a bounding cochain $b$ supported
	on the positive-action double points of $L$. Then $(L, b)$ is
	generated by the Lagrangian cocores of $W$ in $\Wim(W)$.
\end{theorem}

The embedded case is \cite[Theorem~1.1]{cdrgg17}. By
Proposition~\ref{prop:bc-aug} and Theorem~\ref{thm:floer-equivalence},
working with $(L, b)$ in the framework of \cite{gao17} is equivalent to working
with $(L, \eps_b)$ in the framework of \cite{cdrgg17}. The proof follows the
strategy of \cite{cdrgg17} for the most part. We omit the details
that can be found there and focus on the points where the immersed
setting requires additional care.

\begin{proof}[Proof of Theorem \ref{thm:immersed-generation}]
	By Proposition~\ref{prop:bc-aug}, the bounding cochain $b$
	corresponds to an augmentation $\eps_b \colon \CE(L^+) \to k$, and
	by Theorem~\ref{thm:floer-equivalence} the Floer theory of $(L, b)$
	in the framework of \cite{gao17} is canonically identified with the Floer theory
	of $(L, \eps_b)$ in the framework of \cite{cdrgg17}. We therefore carry out the
	argument using the language of augmentations throughout.
	
	Let $x_1, \ldots, x_n$ denote the double points of $L$. We perturb
	$L$ so that it intersects the skeleton of $W$ transversely in
	finitely many points $a_1, \ldots, a_k$, all distinct from
	$x_1, \ldots, x_n$. Let $D_{a_i}$ denote the Lagrangian cocore at
	$a_i$; after a grading shift we may assume $|a_i| = 1$ as a
	morphism from $D_{a_i}$ to $L$. Set
	\[
	L_i := D_{a_i} \quad (i = 1, \ldots, k), \qquad
	L_{k+1} := L, \qquad
	\mathbb{L} := L_1 \cup \cdots \cup L_{k+1}.
	\]
	
	By \cite[Proposition~8.16]{cdrgg17}, if $\CE(\mathbb{L}^+)$ admits
	an augmentation $\eps$ satisfying:
	\begin{enumerate}
		\item $\eps(\bar{x}_l) = \eps_b(\bar{x}_l)$ for
		$l = 1, \ldots, n$,
		\item $\eps(\bar{a}_i) = 1$ for $i = 1, \ldots, k$,
		\item $\eps(\bar{q}) = 0$ for any Reeb chord $\bar{q}$ from
		$L_i^+$ to $L_j^+$ with $i > j$ and $f_i(q) > f_j(q)$,
	\end{enumerate}
	then, denoting by $\mathbb{L}(a_1, \ldots, a_k)$ the immersed exact
	Lagrangian obtained from $\mathbb{L}$ by Lagrangian surgery along
	$a_1, \ldots, a_k$ as in \cite[Section~8.2]{cdrgg17}, the object
	$(\mathbb{L}(a_1, \ldots, a_k), \bar{\eps})$ is quasi-isomorphic to
	a twisted complex built from $D_{a_1}, \ldots, D_{a_k}$ and
	$(L, \eps_b)$. Since $\mathbb{L}(a_1, \ldots, a_k)$ is disjoint
	from the skeleton of $W$, it is equivalent to the zero object by
	\cite[Proposition~7.6]{cdrgg17}. Consequently, $(L, b)$ is
	quasi-isomorphic to a twisted complex built from the cocores
	$D_{a_1}, \ldots, D_{a_k}$ in $\Wim(W)$.
	
	It therefore remains to construct an augmentation $\eps$ of
	$\CE(\mathbb{L}^+)$ satisfying (1)--(3). When $L$ is embedded this
	is \cite[Lemmas~9.4 and~9.5]{cdrgg17}. We extend their inductive
	construction to the immersed setting.
	
	\medskip\noindent\textbf{Setup.}
	For an embedded $L$, \cite{cdrgg17} took suitable specific
	Hamiltonian perturbations of the cocores $D_{a_i}$ and $L$. By
	abuse of notation, we let $L$ denote the Lagrangian after
	perturbation. Let $D_{a_i}^w$ denote the cocore disks after
	perturbation (the superscript $w$ emphasises that the perturbation
	can be viewed as a wrapping), and set
	\[
	\mathbb{L}^w := D_{a_1}^w \cup \cdots \cup D_{a_k}^w \cup L.
	\]
	By taking specific perturbations, \cite{cdrgg17} arranged that near
	each $a_i$ the Lagrangian $L$ looks like a cocore $D_{a_i}$, and
	that every Reeb chord of $(\mathbb{L}^w)^+$ is of one of the
	following types:
	\begin{enumerate}[label=(\alph*)]
		\item chords $\bar{a}_i$ from $(D_{a_i}^w)^+$ to $L^+$, for
		$i = 1, \ldots, k$, of arbitrarily small length $\epsilon > 0$;
		\item chords $\bar{b}^m_{ij}$ from $(D_{a_i}^w)^+$ to $L^+$,
		for $1 \leq i < j \leq k$ and $1 \leq m \leq m_0(i,j)$ for
		some $m_0(i,j) \geq 0$;
		\item chords $\bar{c}^m_{ij}$ from $(D_{a_i}^w)^+$ to
		$(D_{a_j}^w)^+$, for $1 \leq i < j \leq k$ and
		$1 \leq m \leq m_0(i,j)$;
		\item order-reversing chords: chords from $L^+$ to
		$(D_{a_i}^w)^+$, or from $(D_{a_i}^w)^+$ to $(D_{a_j}^w)^+$
		with $i > j$.
	\end{enumerate}
	For an immersed $L$, we take the same perturbation. Every Reeb
	chord of $(\mathbb{L}^w)^+$ is then either of one of the above
	four types, or of the following additional type:
	\begin{enumerate}[label=(\alph*)]
		\setcounter{enumi}{4}
		\item chords $\bar{x}_l$ from $L^+$ to $L^+$ corresponding to
		the double points $x_l$ of $L$, for $l = 1, \ldots, n$.
	\end{enumerate}
	Let $\mathcal{I}$ denote the bilateral ideal generated by
	order-reversing chords. Since $\mathcal{I}$ is preserved by the
	differential, the quotient algebra
	\[
	\mathcal{A} := \CE\bigl((\mathbb{L}^w)^+\bigr)/\mathcal{I}
	\]
	inherits a differential. Moreover, there is a filtration
	\[
	\CE(L^+) =: \mathcal{A}_{k+1} \subset \mathcal{A}_k \subset
	\cdots \subset \mathcal{A}_1 = \mathcal{A},
	\]
	where $\mathcal{A}_i$ is generated by $\bar{a}_s$, $\bar{b}^m_{sj}$,
	$\bar{c}^m_{sj}$ with $s \geq i$, together with all
	$\bar{x}_1, \ldots, \bar{x}_n$. In particular, $\mathcal{A}_{k+1}$
	is generated by $\bar{x}_1, \ldots, \bar{x}_n$; in the embedded
	case $\mathcal{A}_{k+1} = k$. The differential on $\mathcal{A}$
	respects this filtration since every summand of the differential
	of an order-respecting chord is a product of order-respecting
	chords.
	
	The action $a(\bar{p}) = f(p^+) - f(p^-)$ of a Reeb chord $\bar p$
	satisfies the energy inequality~\eqref{eq:energy}: whenever
	$d\bar{p} = \lambda\ \bar{p}_1 \cdots \bar{p}_r + \cdots$ for some
	nonzero $\lambda \in k$, we have
	\begin{equation}
		\label{eq:action-inequality}
		a(\bar{p}) > a(\bar{p}_1) + \cdots + a(\bar{p}_r).
	\end{equation}
	For each fixed index $i$, this induces a total ordering $<_i$ on
	pairs $(j,m)$ by declaring
	\[
	(h,l) <_i (j,m) \quad \text{if and only if} \quad
	a(\bar{c}^l_{ih}) < a(\bar{c}^m_{ij}).
	\]
	The actions of $\bar{b}^m_{ij}$ and $\bar{c}^m_{ij}$ can be made
	pairwise distinct for all $i,j,m$ and sufficiently close to one
	another for each fixed $i,j,m$. The Reeb chords of type (e) can
	be assumed to have positive action by a small perturbation of $L$
	near its double points.
	
	With respect to the ordering $<_i$, the differentials take the
	form
	\begin{align}
		d\bar{a}_i &= 0, \label{eq:diff-a}\\
		d\bar{b}^m_{ij} &=
		\alpha^m_j \bar{a}_i
		+ \sum_{(h,l) <_i (j,m)} \beta^{mh}_{jl}\, \bar{b}^l_{ih}
		+ \bar{a}_j \bar{c}^m_{ij}
		+ \sum_{(h,l) <_i (j,m)} w^{ml}_{jh}\, \bar{c}^l_{ih},
		\label{eq:diff-b}\\
		d\bar{c}^m_{ij} &=
		\sum_{(h,l) <_i (j,m)} \widetilde{w}^{ml}_{jh}\,
		\bar{c}^l_{ih}, \label{eq:diff-c}
	\end{align}
	where $\alpha^m_j, \beta^{mh}_{jl} \in \CE(L^+)$ and
	$w^{ml}_{jh}, \widetilde{w}^{ml}_{jh} \in \mathcal{A}_{i+1}$.
	Since $a(\bar{a}_i) = \epsilon$ can be made arbitrarily small
	by shifting the potentials of $D_{a_i}$,
	equation~\eqref{eq:action-inequality} implies $d\bar{a}_i = 0$.
	The coefficient of $\bar{a}_j \bar{c}^m_{ij}$ equals $1$ by the
	controlled Hamiltonian perturbations; see
	\cite[Lemma~9.4]{cdrgg17} for details. The differential preserves
	$\mathcal{A}_{k+1}$, so the differentials of type (e) chords are
	products of type (e) chords. The only difference from the embedded
	case is that $\alpha^m_j, \beta^{mh}_{jl}$ lie in $\CE(L^+)$
	rather than in $k$.
	
	\medskip\noindent\textbf{Inductive construction.}
	Define an augmentation $\eps'$ of $\mathcal{A}$ inductively. On
	$\mathcal{A}_{k+1} = \CE(L^+)$, set
	\[
	\eps'(\bar{x}_l) := \eps_b(\bar{x}_l), \qquad l = 1, \ldots, n.
	\]
	Assume $\eps'$ is defined on $\mathcal{A}_{i+1}$. Extend it to
	$\mathcal{A}_i$ by setting
	\[
	\eps'(\bar{a}_i) = 1, \qquad
	\eps'(\bar{b}^m_{ij}) = 0 \quad \text{for all } j, m,
	\]
	and defining $\eps'(\bar{c}^m_{ij})$ by induction on
	$<_i$:
	\[
	\eps'(\bar{c}^m_{ij})
	= -\eps'(\alpha^m_j)
	- \sum_{(h,l) <_i (j,m)}
	\eps'(w^{ml}_{jh})\, \eps'(\bar{c}^l_{ih}).
	\]
	It is straightforward to verify that $\eps' \circ d = 0$ on all
	generators except possibly $\bar{x}_l$ and $\bar{c}^m_{ij}$.
	Since $\eps_b$ is an augmentation of $\CE(L^+)$ by
	Proposition~\ref{prop:bc-aug}, we have $\eps'(d\bar{x}_l) = \eps_b(d\bar{x}_l) =0$
	for all $l = 1, \ldots, n$. It remains to check $\bar{c}^m_{ij}$.
	Assume the statement holds on $\mathcal{A}_{i+1}$ and for all
	$\bar{c}^l_{ih}$ with $(h,l) <_i (j,m)$, and consider
	\[
	0 = d(d\bar{b}^m_{ij})
	= d\!\left(\alpha^m_j \bar{a}_i
	+ \sum_{(h,l) <_i (j,m)} \beta^{mh}_{jl}\, \bar{b}^l_{ih}
	+ \bar{a}_j \bar{c}^m_{ij}
	+ \sum_{(h,l) <_i (j,m)} w^{ml}_{jh}\, \bar{c}^l_{ih}\right).
	\]
	Since $\eps'(d\alpha^m_j) = \eps'(d\beta^{mh}_{jl}) = 0$,
	$\eps'(\bar{a}_j) = 1$, $\eps'(d\bar{b}^l_{ih}) = 0$, and
	$d\bar{a}_i = 0$, applying $\eps'$ yields
	\[
	0 = \eps'(d\bar{c}^m_{ij})
	+ \sum_{(h,l) <_i (j,m)}
	\eps'\!\left(d(w^{ml}_{jh}\, \bar{c}^l_{ih})\right).
	\]
	The induction hypothesis gives
	$\eps'(d(w^{ml}_{jh}\, \bar{c}^l_{ih})) = 0$, so
	$\eps'(d\bar{c}^m_{ij}) = 0$, completing the induction. Therefore
	$\eps'$ is an augmentation of $\mathcal{A}$.
	
	\medskip\noindent\textbf{Conclusion.}
	Precomposing $\eps'$ with the projection $\CE(\mathbb{L}^+)
	\twoheadrightarrow \mathcal{A}$ yields an augmentation $\eps$ of
	$\CE(\mathbb{L}^+)$ satisfying (1)--(3). By
	Proposition~\ref{prop:bc-aug}, $\eps$ corresponds to a bounding
	cochain in the framework of \cite{gao17}. The twisted complex built from
	$D_{a_1}, \ldots, D_{a_k}$ and $(L, b)$ is then quasi-isomorphic
	to $(\mathbb{L}(a_1, \ldots, a_k), \bar{\eps})$, which is
	equivalent to the zero object since $\mathbb{L}(a_1, \ldots, a_k)$
	is disjoint from the skeleton of $W$. Hence $(L, b)$ is generated
	by the cocores $D_{a_1}, \ldots, D_{a_k}$ in $\Wim(W)$.
\end{proof}

\begin{remark}
	The hypothesis that $b$ is supported on positive-action double
	points is equivalent, by Proposition~\ref{prop:bc-aug}, to the
	existence of an augmentation $\eps_b \colon \CE(L^+) \to k$ of
	the Chekanov--Eliashberg algebra of the Legendrian lift $L^+$.
	This is precisely the condition used by \cite{cdrgg17} to
	define Floer theory for immersed Lagrangians, so
	Theorem~\ref{thm:immersed-generation} can be equivalently stated
	as: any augmented immersed Lagrangian $(L, \eps)$ in the sense of
	\cite{cdrgg17} is generated by the Lagrangian cocores of $W$ in $\Wim(W)$.
\end{remark}

\begin{remark}
	If $L$ is not nice in the sense of Remark~\ref{rmk:nice-immersion}, 
	one first perturbs $L$ by a Lagrangian isotopy to a nice immersion $L'$, 
	equips $L'$ with a positive-action bounding cochain $b'$, and applies 
	Theorem~\ref{thm:immersed-generation} to $(L', b')$. The original object 
	$(L, b)$, obtained by transporting $(L', b')$ back along the isotopy, 
	is then generated by the Lagrangian cocores.
\end{remark}

\begin{remark}
	As in \cite[Remark~1.3]{cdrgg17}, exact Lagrangian immersions can
	be enriched with additional structure such as local systems,
	gradings, or spin structures, and the same arguments carry over.
	Moreover, the framework of \cite{gao17} can be enriched with local
	systems following the standard procedure in Floer theory, where the
	$A_\infty$-operations are twisted by holonomy weights along the
	boundaries of the relevant holomorphic curves. Since the
	degeneration results of Section~\ref{subsec:rigidity} depend only
	on the energy inequality and the degree
	constraints,
	they carry over unchanged to this enriched setting. In particular,
	Theorem~\ref{thm:immersed-generation} holds for exact Lagrangian
	immersions equipped with local systems.
\end{remark}

The generation theorem has the following immediate consequence.

\begin{corollary}\label{cor:embedded-fukaya}
	Let $(L, b)$ be an exact Lagrangian immersion equipped with a
	bounding cochain supported on positive-action double points. Then
	$(L, b)$ is an object of the triangulated closure $\mathrm{Tw}\, \cW(W)$ of the embedded wrapped Fukaya category $\cW(W)$ appearing in
	\cite{cdrgg17} or \cite{gps1}.
\end{corollary}

\begin{proof}
	By Theorem~\ref{thm:immersed-generation}, $(L, b)$ is generated
	by the Lagrangian cocores of $W$ in $\Wim(W)$. The cocores are
	embedded Lagrangians, hence objects of the embedded wrapped Fukaya
	category $\cW(W)$, which embeds as a full subcategory of $\Wim(W)$
	by \cite[Proposition~5.22]{gao17}, or alternatively by
	Corollary~\ref{cor:immersed-fukaya}. Since $\mathrm{Tw}\,\cW(W)$ is closed under twisted complexes, and since
	$(L,b)$ is quasi-isomorphic in $\Wim(W)$ to a twisted complex of objects of
	$\cW(W)$, it follows that $(L,b)$ belongs to $\mathrm{Tw}\,\cW(W)$.
\end{proof}

\bibliographystyle{amsalpha}
\bibliography{Bibliography}

\end{document}